\documentclass[a4paper,11pt]{article}
\usepackage[english]{babel}
\usepackage[latin1]{inputenc}
\usepackage{amssymb}
\usepackage{amscd}
\usepackage{mathtools}
\usepackage{enumerate}
\usepackage[paper=a4paper,left=3cm,right=3cm,top=3cm,bottom=2.3cm]{geometry}
\usepackage[thmmarks,amsmath,amsthm]{ntheorem}

\usepackage{hyperref}

\newtheorem{thm}{Theorem}[section]
\newtheorem{lem}[thm]{Lemma}
\newtheorem{defi}[thm]{Definition}

\newtheorem{cor}[thm]{Corollary}
\newtheorem{rem}[thm]{Remark}

\begin{document}

\title{The symplectic structure on the moduli space of line bundles on a noncommutative Azumaya surface}
\author{Fabian Reede\thanks{Mary Immaculate College, Limerick, fabianreede@gmx.net\newline This research was supported by the Deutsche Forschungsgemeinschaft (DFG) with a research fellowship}}
\maketitle

\begin{abstract}
In this note we prove that the moduli space of torsion-free modules of rank one over an Azumaya algebra on a $K3$-surface is an irreducible symplectic variety deformation equivalent to a Hilbert scheme of points on the $K3$-surface. 
\end{abstract}
\section*{Introduction}
Assume $X$ is a smooth projective $K3$-surface over $\mathbb{C}$ and let $\mathcal{A}$ be an Azumaya algebra on $X$. The algebra $\mathcal{A}$ is an example of a so called noncommutative projective surface and also an example of what is called a Calabi-Yau order, that is a noncommutative analogue of a classical $K3$-surface.
\medskip

\noindent Locally projective $\mathcal{A}$-modules of rank one can be considered as line bundles on this noncommutative surface. In (\cite{hst}) the authors construct moduli schemes for such line bundles. These schemes can be seen as noncommutative versions of the usual Picard schemes. By allowing torsion-free $\mathcal{A}$-modules and by fixing invariants, for example the Mukai vector, theses moduli schemes are shown to be projective schemes over $\mathbb{C}$. Furthermore the authors show that these moduli schemes are smooth and possess a symplectic structure. So one can ask the question: are these moduli schemes irreducible symplectic varieties (hyperk{\"a}hler manifolds)?
\medskip

\noindent There are 4 known classes of hyperk{\"a}hler manifolds: we  have the Hilbert schemes of points on a smooth projective $K3$-surface and the generalized Kummer varieties associated to an abelian surface, both of these classes are due to Beauville. Furthermore there is a class of 6-dimensional examples and one class of 10-dimensional examples, both classes are due to O'Grady using symplectic desingularization. All other known examples of hyperk{\"a}hler manifolds are deformation equivalent to one of the four examples mentioned above. So if the moduli spaces of line bundles on a noncommutative Azumaya surface are  hyperk{\"a}hler manifolds, do they give rise to a new class or are they also deformation equivalent to one of the four known classes?
\medskip

\noindent In this note we prove that these moduli spaces are in fact hyperk{\"a}hler manifolds and that they do not give rise to a new class of hyperk{\"a}hler manifolds. We prove that they belong to the first class by showing that they are deformation equivalent to appropriate Hilbert schemes. This is done by comparing the moduli spaces to moduli spaces constructed by Yoshioka in (\cite{yoshi}). Yoshioka proves that they are hyperk{\"a}hler manifolds deformation equivalent to Hilbert schemes, hence so are the moduli spaces of line bundles on a noncommutative Azumaya surface.
\tableofcontents

\section{Modules over Azumaya algebras and Brauer-Severi varieties}
In this section we assume $X$ is any fixed scheme of finite type over $\mathbb{C}$, that is $X\in Sch_{\mathbb{C}}$. Following (\cite{lehn}) we denote by $Sch_{\mathbb{C}}$ the category of schemes of finite type over $\mathbb{C}$ and for $X\in Sch_{\mathbb{C}}$ we denote the category of schemes of finite type over $X$ by $Sch_X$.   
\begin{defi}
A sheaf $\mathcal{A}$ of $\mathcal{O}_X$-algebras is called Azumaya algebra if for every closed point $x\in X$ the fiber $\mathcal{A}(x)=\mathcal{A}\otimes k(x)$ is a central simple algebra over the residue field $k(x)$.
\end{defi}
\begin{rem}
\normalfont The rank of an Azumaya algebra $\mathcal{A}$ is always a square, that is we have $rk(\mathcal{A})=r^2$. We also note that for every morphism $f: T\rightarrow X$ the $\mathcal{O}_T$-algebra $f^{*}\mathcal{A}$ is an Azumaya algebra on $T$.
\end{rem}
Define the functor 
\begin{equation*}
\mathcal{B}\mathcal{S}(\mathcal{A}): (Sch_X)^{op} \rightarrow Sets
\end{equation*}
which sends an $X$-scheme $f: T\rightarrow X$ to the set of left ideals $\mathcal{I}\subseteq \mathcal{A}_T$ such that $\mathcal{A}_T/\mathcal{I}$ is a locally free $\mathcal{O}_T$-module of rank $r(r-1)$, here $\mathcal{A}_T=f^{*}\mathcal{A}$.
\begin{lem}
The functor $\mathcal{B}\mathcal{S}(\mathcal{A})$ is representable by an $X$-scheme $\pi: BS(\mathcal{A})\rightarrow X$. The morphism $\pi$ is faithfully flat and proper and exhibits $BS(\mathcal{A})$ as an \'etale $\mathbb{P}^{r-1}$-bundle over $X$.
\end{lem}
\begin{proof}
For the first assertion see (\cite[8.13, Exercise 8.14.]{wedh}) and for the second assertion see (\cite[Th\'eor\`eme 8.2, Corollaire 8.3]{groth}).
\end{proof}
\begin{rem}\label{functor}
\normalfont  For every $X$-scheme $f: T\rightarrow X$ we get by functoriality a canonical isomorphism
\begin{equation*}
BS(\mathcal{A}_T)=BS(f^{*}\mathcal{A})\cong BS(\mathcal{A})\times_X T.
\end{equation*}
\end{rem}
The scheme $BS(\mathcal{A})$ is called the Brauer-Severi variety associated to $\mathcal{A}$. In the following we will just write $\pi: Y\rightarrow X$ for this $X$-scheme.\\
There is a canonically defined locally free sheaf $G$ of rank $r$ on $Y$, which is unique up to scalars (\cite[Lemma 1.1.]{yoshi}), which sits in the Euler sequence
\begin{equation*}\begin{CD}
0 @>>> \mathcal{O}_Y @>>>  G @>>> \mathcal{T}_{Y/X} @>>> 0.
\end{CD}\end{equation*}
Using ideas of Yoshioka, see (\cite[Lemma 1.5.]{yoshi}), we define the following subcategory of $Coh(Y)$:
\begin{defi}
Assume $\pi: Y\rightarrow X$ is a Brauer-Severi variety, then we define:
\begin{equation*}
Coh(Y,X):=\{E\in Coh(Y) | \pi^{*}\pi_{*}(E\otimes G^{\vee})\xrightarrow{\hspace{0.2cm}\sim \hspace{0.2cm}}E\otimes G^{\vee}\}.
\end{equation*}
Here the morphism is the canonical morphism coming from the adjunction between $\pi^{*}$ and $\pi_{*}$.
\end{defi}
Furthermore we denote the category of coherent sheaves on $X$ having the structure of a left respectively right $\mathcal{A}$-module by $Coh_l(X,\mathcal{A})$ respectively $Coh_r(X,\mathcal{A})$.
\begin{lem}\label{directim}
Assume $\pi: Y\rightarrow X$ is a Brauer-Severi variety, then for $E\in Coh(X)$ we have
\begin{equation*}
R^i\pi_{*}(\pi^{*}E)\cong \begin{cases} E &\mbox{if } i = 0 \\
0 & \mbox{if } i \geq 1. \end{cases}
\end{equation*}
\end{lem}
\begin{proof}
For every $i\geq 0$ there is a natural morphism $R^i\pi_{*}(\mathcal{O}_Y)\otimes E\rightarrow R^i\pi_{*}(\pi^{*}E)$. Since $Y$ is an \'etale $\mathbb{P}^{r-1}$-bundle over $X$ we have
\begin{equation*}
R^i\pi_{*}\mathcal{O}_Y\cong \begin{cases} \mathcal{O}_X &\mbox{if } i = 0 \\
0 & \mbox{if } i \geq 1. \end{cases}
\end{equation*}
It is therefore enough to prove this after a faithfully flat \'etale base change $f: X'\rightarrow X$ that splits $\mathcal{A}$.
\begin{equation*}\begin{CD}
Y' @>{f'}>> Y\\
@V{\pi'}VV @VV{\pi}V\\
X' @>{f}>> X
\end{CD}\end{equation*}
So we have $f^{*}\mathcal{A}\cong\mathcal{E}nd_{\mathcal{O}_{X'}}(\mathcal{E})$ for a locally free sheaf $\mathcal{E}$ on $X'$. But then by (\ref{functor})
\begin{equation*}
Y'=Y\times_X X'\cong BS(f^{*}\mathcal{A})
\end{equation*}
and using Morita equivalence we see that $BS(f^{*}\mathcal{A})=BS(\mathcal{E}nd_{\mathcal{O}_{X'}}(\mathcal{E}))\cong \mathbb{P}(\mathcal{E}^{\vee})$. Now the result follows from (\cite[4.5.(f)]{thoma}).
\end{proof}
\begin{rem}\label{asplit}
\normalfont The Brauer-Severi variety $\pi: Y\rightarrow X$ associated to an Azumaya algebra $\mathcal{A}$ has the following splitting property: there is an isomorphism of $\mathcal{O}_Y$-algebras on $Y$
\begin{equation*}
\pi^{*}\mathcal{A}\cong \mathcal{E}nd_{\mathcal{O}_Y}(G)^{op},
\end{equation*}
see (\cite[8.4.]{quillen}). This shows that we have $\pi^{*}(\mathcal{A}^{op})\cong \mathcal{E}nd_{\mathcal{O}_Y}(G)$ and hence  
\begin{equation*}
\mathcal{A}^{op}\cong \pi_{*}\pi^{*}(\mathcal{A}^{op})\cong \pi_{*}(\mathcal{E}nd_{\mathcal{O}_Y}(G)).
\end{equation*}
\end{rem}
\begin{rem}
\normalfont Using Morita equivalence one can see that there is an equivalence of categories
\begin{equation*}
Coh_r(Y,\mathcal{E}nd_{\mathcal{O}_Y}(G))\cong Coh(Y).
\end{equation*}
The equivalence is given by $E\mapsto E\otimes_{\mathcal{E}nd_{\mathcal{O}_Y}(G)}G$ with inverse $F\mapsto F\otimes G^{\vee}$, see for example (\cite[8.12.]{wedh}). 
\end{rem}
\begin{lem}\label{equiv}
Assume $\mathcal{A}$ is an Azumaya algebra on $X$ and let $\pi: Y\rightarrow X$ be the associated Brauer-Severi variety. Then there is an equivalence of categories:
\begin{equation*}
Coh_l(X,\mathcal{A})\cong Coh(Y,X).
\end{equation*}
\end{lem}
\begin{proof}
The category $Coh_l(X,\mathcal{A})$ is isomorphic to $Coh_r(X,\mathcal{A}^{op})$, hence it is enough to show that $Coh_r(X,\mathcal{A}^{op})$ and $Coh(Y,X)$ are equivalent. For this we define the following functors:
\begin{alignat*}{2}
F&: Coh_r(X,\mathcal{A}^{op})\rightarrow Coh(Y,X), E&&\mapsto \pi^{*}E\otimes_{\pi^{*}(\mathcal{A}^{op})}G\\
H&: Coh(Y,X)\rightarrow Coh_r(X,\mathcal{A}^{op}), E&&\mapsto \pi_{*}(E\otimes G^{\vee})
\end{alignat*}
First of all we need to verify that $F$ and $H$ are well-defined. That is we need to see that
\begin{enumerate}[i)]
\item $\pi^{*}E\otimes_{\pi^{*}(\mathcal{A}^{op})}G \in Coh(Y,X)$ 
\item $\pi_{*}(E\otimes G^{\vee})\in Coh_r(X,\mathcal{A}^{op})$.
\end{enumerate}
For i) we look at the canonical morphism
\begin{equation*}
\pi^{*}\pi_{*}((\pi^{*}E\otimes_{\pi^{*}(\mathcal{A}^{op})}G)\otimes G^{\vee})\rightarrow (\pi^{*}E\otimes_{\pi^{*}(\mathcal{A}^{op})}G)\otimes G^{\vee}.
\end{equation*}
But by definition we have $G\otimes G^{\vee}\cong \mathcal{E}nd_{\mathcal{O}_Y}(G)\cong\pi^{*}(\mathcal{A}^{op})$. So we have to see that $\pi^{*}\pi_{*}\pi^{*}E\rightarrow \pi^{*}E$ is an isomorphism. But this follows since $E\rightarrow \pi_{*}\pi^{*}E$ is an isomorphism by (\ref{directim}) and $(\pi^{*},\pi_{*})$ is a pair of adjoint functors.\\
For ii) we just note that $G^{\vee}$ is a right $\mathcal{E}nd_{\mathcal{O}_Y}(G)$-module, hence so is $E\otimes G^{\vee}$. Then $\pi_{*}(E\otimes G^{\vee})$ is a right $\pi_{*}(\mathcal{E}nd_{\mathcal{O}_Y}(G))\cong \mathcal{A}^{op}$-module by (\ref{asplit}).\\
Now we study $F\circ H$ and $H\circ F$: for $E\in Coh_r(X,\mathcal{A}^{op})$ we have:
\begin{equation*}
H(F(E))=\pi_{*}((\pi^{*}E\otimes_{\pi^{*}(\mathcal{A}^{op})}G)\otimes G^{\vee})\cong \pi_{*}\pi^{*}E\cong E
\end{equation*} 
again by (\ref{directim}). Now for $E\in Coh(Y,X)$ we get
\begin{equation*}
F(H(E))=\pi^{*}(\pi_{*}(E\otimes G^{\vee}))\otimes_{\pi^{*}(\mathcal{A}^{op})}G\cong (E\otimes G^{\vee})\otimes_{\pi^{*}(\mathcal{A}^{op})}G\cong E
\end{equation*}
The first isomorphism follows from $E\in Coh(Y,X)$ and the second isomorphism follows from $G^{\vee}\otimes_{\pi^{*}(\mathcal{A}^{op})}G\cong \mathcal{O}_Y$, see (\cite[Proposition 8.26.]{wedh}).\\
This shows that these two categories are equivalent and therefore $Coh_l(X,\mathcal{A})$ and $Coh(Y,X)$ are also equivalent.
\end{proof}

\section{Moduli spaces of line bundles on a noncommutative Azumaya surface}
In this section we work again with a fixed scheme $X$, which in this case should be a smooth projective $K3$-surface over $\mathbb{C}$. Thus we have the Mukai pairing $<-,->$ on $H^{2*}(X,\mathbb{Q})$ given by 
\begin{equation*}
<x,y>=-\int_X x^{\vee}y\hspace{0.7cm} x,y\in H^{2*}(X,\mathbb{Q}),
\end{equation*}
see for example (\cite[2.6.1.5]{lehn}). To shorten notation we write $x^2$ for the term $<x,x>$.\\
We want to study moduli spaces of sheaves on $X$, so we need to understand families of sheaves. For this we pick an Azumaya algebra $\mathcal{A}$ of rank $r^2$ on $X$ and let $\pi: Y\rightarrow X$ be the associated Brauer-Severi variety.\\ 
For every $S\in Sch_{\mathbb{C}}$ and for every $s\in S$ we have the double pullback diagram: 
\begin{equation*}\begin{CD}
Y_s @>{j_s}>> Y\times S @>{q}>> Y\\
@V{\pi_s}VV @V{\pi_S}VV @VV{\pi}V\\
X_s @>{i_s}>> X\times S @>{p}>> X
\end{CD}\end{equation*}
where we have $Y\times S\cong BS(\mathcal{A}_S)$ and $Y_s\cong BS(\mathcal{A}_s)$ by (\ref{functor}) so that both schemes are also Brauer-Severi varieties.\\
The following objects live on the various schemes in this diagram: 
\begin{itemize}
\item $\mathcal{A}$ on $X$, $\mathcal{A}_S:=p^{*}\mathcal{A}$ on $X\times S$ and $\mathcal{A}_s:=i_s^{*}p^{*}\mathcal{A}$ on $X_s$.
\item $\pi^{*}(\mathcal{A}^{op})\cong\mathcal{E}nd_{\mathcal{O}_{Y}}(G)$ on $Y$, $\pi_S^{*}(\mathcal{A}_S^{op})\cong\mathcal{E}nd_{\mathcal{O}_{Y\times S}}(G_S)$ on $Y\times S$ with $G_S=q^{*}G$ and $\pi_s^{*}(\mathcal{A}_s^{op})\cong\mathcal{E}nd_{\mathcal{O}_{Y_s}}(G_s)$ on $Y_s$ with $G_s=j_s^{*}q^{*}G$.
\end{itemize}
\begin{rem}
\normalfont At the generic point $\eta\in X$ the Azumaya algebra $\mathcal{A}$ is given by the central simple algebra $\mathcal{A}_\eta=M_n(D)$ for some division ring $D$ over the function field $\mathbb{C}(X)$. Without loss of generality we may assume $n=1$. This is because we are only interested in the Brauer class $[\mathcal{A}]\in Br(X)$. (Brauer equivalent algebras have equivalent module categories and hence isomorphic moduli spaces.) But we have $[M_n(D)]=[D]\in Br(\mathbb{C}(X))$ so by injectivity of $Br(X)\rightarrow Br(\mathbb{C}(X))$, $[\mathcal{A}]\mapsto [\mathcal{A}_\eta]$ all we need to do is find an Azumaya algebra $\mathcal{D}$ on $X$ with $\mathcal{D}_\eta=D$. The existence of such an algebra is due the following theorem:
\end{rem}
\begin{thm}
Assume $X$ is a regular integral scheme of dimension at most two with function field $K$. If $A$ is a central simple $K$-algebra of dimension $n^2$ whose image in $_nBr(K)$ is unramified at every point of codimension one in $X$, then there is an Azumaya algebra $\mathcal{A}$ of rank $n^2$ on $X$ such that $\mathcal{A}\otimes K=A$.  
\end{thm}
\begin{proof}
See (\cite[Th\'eor\`eme 2.5.]{col}).	
\end{proof}
Using this we see that a generically simple $\mathcal{A}$-module $E$, i.e. $E_\eta$ is a simple $\mathcal{A}_\eta$-module, must be generically of rank one over $\mathcal{A}$, hence has rank $r^2$ over $X$. We call such modules $\mathcal{A}$-modules of rank one. Now we can study the moduli functors of interest.\\
First we define
\begin{equation*}
\mathcal{M}_{\mathcal{A}/X}(v_\mathcal{A}): (Sch_{\mathbb{C}})^{op}\rightarrow Sets
\end{equation*}
which sends a $\mathbb{C}$-scheme $S$ to the set of isomorphism classes of families of torsion-free $\mathcal{A}$-modules of rank one with Mukai vector $v_\mathcal{A}$ over $S$.\\
The following facts are known:
\begin{thm}
The moduli functor $\mathcal{M}_{\mathcal{A}/X}(v_\mathcal{A})$ has a coarse moduli scheme $M_{\mathcal{A}/X}(v_\mathcal{A})$. Furthermore $M_{\mathcal{A}/X}(v_\mathcal{A})$ is a smooth projective scheme with a symplectic form on its tangent bundle.
\end{thm}
\begin{proof}
See (\cite[Theorem 2.4., Theorem 3.6.]{hst}).
\end{proof}
The other moduli functor of interest is the following:
\begin{equation*}
\mathcal{M}^{Y,G}_{H}(v_G): (Sch_{\mathbb{C}})^{op}\rightarrow Sets
\end{equation*}
which maps a $\mathbb{C}$-scheme $S$ to the set of isomorphism classes of families of torsion-free $G$-twisted semistable $\mathcal{O}_Y$-modules of rank $r$ with Mukai vector $v_G$ over $S$.\\
This moduli functor has also been well studied:
\begin{thm}\label{yosh}
The moduli functor $\mathcal{M}^{Y,G}_{H}(v_G)$ has a coarse moduli scheme $M^{Y,G}_{H}(v_G)$. If $v_G$ is primitive and $H$ is general with respect to $v_G$, then all $G$-twisted semistable sheaves are $G$-twisted stable. In particular $M^{Y,G}_{H}(v_G)$ is a smooth projective scheme with a symplectic form on its tangent bundle.
\end{thm}
\begin{proof}
See (\cite[Theorem 2.1.,Proposition 3.6., Theorem 3.11.]{yoshi}).
\end{proof}
\begin{rem}\label{mukai}
\normalfont By definition, (\cite[Definition 2.7]{reede}), we have for $E\in Coh_l(X,\mathcal{A})$: 
\begin{equation*}
v_\mathcal{A}(E)=\frac{ch(E)}{\sqrt{ch(\mathcal{A})}}\sqrt{td(X)}.
\end{equation*}
Furthermore in (\cite[Definition 3.1.]{yoshi}) Yoshioka defined for $E\in Coh(Y,X)$:
\begin{equation*}
v_G(E)=\frac{ch(\pi_{*}(E\otimes G^{\vee}))}{\sqrt{ch(\pi_{*}(G\otimes G^{\vee}))}}\sqrt{td(X)}.
\end{equation*}
(Actually Yoshioka defines this vector using the derived direct image, but the sheaf $E\otimes G^{\vee}$ does not have higher direct images for $E\in Coh(Y,X)$ by (\ref{directim}).)\\
Using the equivalence from (\ref{equiv}) we have $v_G(F(E))=v_\mathcal{A}(E)$ and $v_\mathcal{A}(H(E))=v_G(E)$. This is because  $\mathcal{A}^{op}\cong\pi_{*}(G\otimes G^{\vee})$ and $ch(\mathcal{A}^{op})=ch(\mathcal{A})$. So these Mukai vectors are the same and in the following we will omit the subscript and just write $v$ for a fixed Mukai vector.
\end{rem}
If $E\in Coh_l(X,\mathcal{A})$ is given such that $E$ is a torsion-free $\mathcal{A}$-module of rank one with Mukai vector $v=v_\mathcal{A}(E)$, then we have $rk(E)=r^2$. Now $E$ has no $\mathcal{A}$-submodules $E'\subsetneq E$ with $0<rk(E')<r^2$ because any such module must satisfy $r^2|rk(E')$ which is impossible. So $F(E)$ is a torsion-free $\mathcal{O}_Y$-module of rank $r$, has Mukai vector $v_G=v$ and has no submodules in $Coh(Y,X)$ since $F$ preserves submodules. $F(E)$ is therefore $G$-twisted stable.\\
On the other hand if $E\in Coh(Y,X)$ is given and $E$ is a torsion-free $G$-twisted semistable $\mathcal{O}_Y$-module of rank $r$ with $v_G(E)=v$ then $H(E)\in Coh_l(X,\mathcal{A})$ and $H(E)$ is torsion-free and of rank one, as $rk(H(E))=r^2$. We can conclude that $E$ must be $G$-twisted stable, because any submodule $E'\subsetneq E$ in $Coh(Y,X)$ with $0<rk(E')<r$ would give rise to an $\mathcal{A}$-submodule $H(E')\subsetneq H(E)$ in $Coh_l(X,\mathcal{A})$ with $0<rk(H(E'))<r^2$ as $H$ also preserves submodules. But this is impossible since $H(E)$ is a torsion-free $\mathcal{A}$-module of rank one.\\
We have thus proven the following lemma:
\begin{lem}\label{stable}
Assume $v$ is a Mukai vector with $v=v_\mathcal{A}(E)$ for some torsion-free $\mathcal{A}$-module $E$ of rank one, then any torsion-free $G$-twisted semistable $\mathcal{O}_Y$-module of rank $r$ with Mukai vector $v_G=v$ is $G$-twisted stable.
Under the equivalence $Coh_l(X,\mathcal{A})\cong Coh(Y,X)$ torsion-free $\mathcal{A}$-modules of rank one with Mukai vector $v$ correspond to torsion-free $G$-twisted stable $\mathcal{O}_Y$-modules of rank $r$ with Mukai vector $v$.
\end{lem}
\begin{rem}
\normalfont The previous lemma (\ref{stable}) shows that given a Mukai vector $v=v_\mathcal{A}(E)$ for some torsion-free $\mathcal{A}$-module $E$, we do not need to worry if the polarization $H$ is general with respect to $v$, see (\ref{yosh}). All torsion-free $G$-twisted semistable $\mathcal{O}_Y$-modules of rank $r$ and with Mukai vector $v_G=v$ are automatically $G$-twisted stable. So $M_H^{Y,G}(v)$ is a smooth projective scheme without choosing a special polarization $H$.  
\end{rem}
Looking at the definition of the moduli functors, we see that we are working with two kind of families of sheaves:
\begin{enumerate}[i)]
\item a family of torsion-free $\mathcal{A}$-modules of rank one with Mukai vector $v$ is a sheaf $\mathcal{E}\in Coh_l(X\times S,\mathcal{A}_S)$ such that $\mathcal{E}$ is flat over $S$ and $\mathcal{E}_s=i_s^{*}\mathcal{E}$ is a torsion-free $\mathcal{A}_s$-module of rank one on $X_s$ with Mukai vector $v$ for every closed point $s\in S$, especially $\mathcal{E}_s\in Coh_l(X_s,\mathcal{A}_s)$.
\item a family of torsion-free $G$-twisted semistable $\mathcal{O}_Y$-modules of rank $r$ with Mukai vector $v$ is a sheaf $\mathcal{F}\in Coh(Y\times S,X\times S)$ such that $\mathcal{F}$ is flat over $S$ and $\mathcal{F}_s=j_s^{*}\mathcal{F}$ is a torsion-free $G$-twisted semistable $\mathcal{O}_Y$-module of rank $r$ on $Y_s$ with Mukai vector $v$ for every closed point $s\in S$, especially $\mathcal{F}_s\in Coh(Y_s,X_s)$.
\end{enumerate}
Now for every $S\in Sch_{\mathbb{C}}$ we have $X\times S\in Sch_{\mathbb{C}}$. Furthermore $\mathcal{A}_S$ is an Azumaya algebra on $X\times S$ with a canonical isomorphism $BS(\mathcal{A}_S)\cong Y\times S$ so lemma (\ref{equiv}) gives us an equivalence of categories
\begin{equation*}
Coh_l(X\times S,\mathcal{A}_S)\cong Coh(Y\times S,X\times S)
\end{equation*}
which has the following property:
\begin{lem}\label{famil}
The equivalence $Coh_l(X\times S,\mathcal{A}_S)\cong Coh(Y\times S,X\times S)$ maps families of type i) to families of type ii) with semistable replaced by stable and vice versa.
\end{lem}
\begin{proof}
Let $\mathcal{E}$ be a family of type i). Define $\mathcal{F}:=\pi_S^{*}\mathcal{E}\otimes_{\pi_S^{*}(\mathcal{A}_S^{op})}G_S$, then $\mathcal{F}$ is a family of type ii). We have $\mathcal{F}\in Coh(Y\times S,X\times S)$ and $\mathcal{F}$ is flat over $S$. To see this note that $\pi_S$ is faithfully flat, so $\pi_S^{*}\mathcal{E}$ is flat over $S$, see (\cite[2.2.11 (iii)]{ega}). Furthermore $G_S$ is a flat $\pi_S^{*}(\mathcal{A}_S^{op})$-module, so $\mathcal{F}$ is flat over $S$. We also see that
\begin{equation*}
\mathcal{F}_s=j_s^{*}(\pi_S^{*}\mathcal{E}\otimes_{\pi_S^{*}(\mathcal{A}_S^{op})}G_S)=\pi_s^{*}\mathcal{E}_s\otimes_{\pi_s^{*}(\mathcal{A}_s^{op})}G_s. 
\end{equation*} 
So $\mathcal{F}_s\in Coh(Y_s,X_s)$ and $\mathcal{F}_s$ is a torsion-free $G$-twisted stable $\mathcal{O}_Y$-module of rank $r$ on $Y_s$ with Mukai vector $v$ by (\ref{stable}).\\
Let $\mathcal{F}$ be a family of type ii). Define $\mathcal{E}:={\pi_S}_{*}(\mathcal{F}\otimes G_S^{\vee})$, then $\mathcal{E}$ is family of type i). We have $\mathcal{E}\in Coh_l(X\times S,\mathcal{A}_S)$ and $\mathcal{E}$ is flat over $S$. To see this we note that this can be tested after pullback with a faithfully flat morphism $f:Z\rightarrow X\times S$ by (\cite[2.2.11 (iii)]{ega}), so we just use $f=\pi_S$. But $\pi_S^{*}\mathcal{E}=\pi_S^{*}{\pi_S}_{*}(\mathcal{F}\otimes G_S^{\vee})\cong \mathcal{F}\otimes G_S^{\vee}$ and the latter is flat over $S$ since $\mathcal{F}$ and $G_S^{\vee}$ are. Finally 
\begin{equation*}
\begin{split}
\mathcal{E}_s&=i_s^{*}{\pi_S}_{*}(\mathcal{F}\otimes G_S^{\vee})\cong {\pi_s}_{*}\pi_s^{*}i_s^{*}{\pi_S}_{*}(\mathcal{F}\otimes G_S^{\vee})\\
&={\pi_s}_{*}j_s^{*}\pi_S^{*}{\pi_S}_{*}(\mathcal{F}\otimes G_S^{\vee})\cong {\pi_s}_{*}j_s^{*}(\mathcal{F}\otimes G_S^{\vee})={\pi_s}_{*}(\mathcal{F}_s\otimes G_s^{\vee}).
\end{split}
\end{equation*}
So $\mathcal{E}_s\in Coh_l(X_s,\mathcal{A}_s)$ and $\mathcal{E}_s$ is a torsion-free $\mathcal{A}_s$-module of rank one on $X_s$ with Mukai vector $v$ by (\ref{stable}).
\end{proof}
\begin{thm}
Assume $v=v_\mathcal{A}(E)$ is a Mukai vector for some torsion-free $\mathcal{A}$-module $E$ of rank one, then the functors $\mathcal{M}_{\mathcal{A}/X}(v)$ and $\mathcal{M}^{Y,G}_{H}(v)$ are isomorphic.  
\end{thm}
\begin{proof}
We define a natural transformation $\eta: \mathcal{M}_{\mathcal{A}/X}(v)\rightarrow \mathcal{M}^{Y,G}_{H}(v)$:\\
For $S\in Sch_{\mathbb{C}}$ we let $\eta_S: \mathcal{M}_{\mathcal{A}/X}(v)(S)\rightarrow \mathcal{M}^{Y,G}_{H}(v)(S)$ be the map 
\begin{equation*}
[\mathcal{E}]\mapsto [\pi_S^{*}\mathcal{E}\otimes_{\pi_S^{*}(\mathcal{A}_S^{op})}G_S].
\end{equation*}
This map is well-defined by (\ref{famil}). Let $T\in Sch_{\mathbb{C}}$ with $f: T\rightarrow S$. Then we have 
\begin{align*}
\mathcal{M}_{\mathcal{A}/X}(v)(f)&: \mathcal{M}_{\mathcal{A}/X}(v)(S)\rightarrow \mathcal{M}_{\mathcal{A}/X}(v)(T),\hspace{0.1cm} [\mathcal{E}]\mapsto [f^{*}_X(\mathcal{E})]\\
\mathcal{M}^{Y,G}_{H}(v)(f)&: \mathcal{M}^{Y,G}_{H}(v)(S)\rightarrow \mathcal{M}^{Y,G}_{H}(v)(T),\hspace{0.35cm} [\mathcal{F}]\mapsto [f^{*}_Y(\mathcal{F})]
\end{align*}
where $f_X=id_X\times f$ and $f_Y=id_Y\times f$. This gives the diagram:
\begin{equation*}\begin{CD}
Y\times T @>{f_Y}>> Y\times S @>{q}>> Y\\
@V{\pi_T}VV @V{\pi_S}VV @VV{\pi}V\\
X\times T @>{f_X}>> X\times S @>{p}>> X
\end{CD}\end{equation*}
One computes:
\begin{equation*}
\begin{split}
\eta_T(\mathcal{M}_{\mathcal{A}/X}(v)(f)[\mathcal{E}])&=[\pi_T^{*}f_X^{*}\mathcal{E}\otimes_{\pi_T^{*}(\mathcal{A}_T^{op})}G_T]\\
&=[f_Y^{*}\pi_S^{*}\mathcal{E}\otimes_{f_Y^{*}(\pi_S^{*}(\mathcal{A}_S^{op}))}f_Y^{*}G_S]=\mathcal{M}^{Y,G}_{H}(v)(f)(\eta_S([\mathcal{E}]))
\end{split}
\end{equation*}
Hence 
\begin{equation*}
\eta_T\circ\mathcal{M}_{\mathcal{A}/X}(v)(f)=\mathcal{M}^{Y,G}_{H}(v)(f)\circ\eta_S.
\end{equation*}
So $\eta$ is in fact a natural transformation.\\
Define a second natural transformation $\Psi: \mathcal{M}^{Y,G}_{H}(v)\rightarrow \mathcal{M}_{\mathcal{A}/X}(v)$ by 
\begin{equation*}
\Psi_S: \mathcal{M}^{Y,G}_{H}(v)(S)\rightarrow \mathcal{M}_{\mathcal{A}/X}(v)(S),\hspace{0.2cm} [\mathcal{F}]\mapsto [{\pi_S}_{*}(\mathcal{F}\otimes G_S^{\vee})].
\end{equation*}
Again this map is well-defined by (\ref{famil}). Then one sees
\begin{equation*}
\begin{split}
\Psi_T(\mathcal{M}^{Y,G}_{H}(v)(f)[\mathcal{F}])&=[{\pi_T}_{*}(f_Y^{*}\mathcal{F}\otimes G_T^{\vee})]=[{\pi_T}_{*}(f_Y^{*}(\mathcal{F}\otimes G_S^{\vee}))]\\
&=[{\pi_T}_{*}(f_Y^{*}(\pi_S^{*}{\pi_S}_{*}(\mathcal{F}\otimes G_S^{\vee})))]\\
&=[{\pi_T}_{*}\pi_T^{*}(f_X^{*}({\pi_S}_{*}(\mathcal{F}\otimes G_S^{\vee})))]\\
&=[f_X^{*}({\pi_S}_{*}(\mathcal{F}\otimes G_S^{\vee}))]=\mathcal{M}_{\mathcal{A}/X}(v)(f)(\Psi_S([\mathcal{F}]))
\end{split}
\end{equation*}
and hence $\Psi_T\circ\mathcal{M}^{Y,G}_{H}(v)(f)=\mathcal{M}_{\mathcal{A}/X}(v)(f)\circ\Psi_S$.\\
Finally by what we have already seen $\eta_S$ are $\Psi_S$ are inverse bijections for every $S$, so $\eta$ is a natural isomorphism between these moduli functors.
\end{proof}
\begin{cor}\label{iso}
Assume $v=v_\mathcal{A}(E)$ is a Mukai vector for some torsion-free $\mathcal{A}$-module $E$ of rank one, then the moduli spaces $M_{\mathcal{A}/X}(v)$ and $M_H^{Y,G}(v)$ are isomorphic.
\end{cor}
\begin{proof}
Since these schemes are coarse moduli spaces they corepresent their corresponding moduli functors, see (\cite[Definition 2.2.1]{lehn}). The natural isomorphisms $\eta$ and $\Psi$ thus induce natural isomorphisms $Hom_{Sch_{\mathbb{C}}}(-,M_{\mathcal{A}/X}(v))\cong Hom_{Sch_{\mathbb{C}}}(-,M_H^{Y,G}(v))$ by the universal property. But then Yoneda implies $M_{\mathcal{A}/X}(v)\cong M_H^{Y,G}(v)$.
\end{proof}
\begin{cor}\label{main}
Assume $v=v_\mathcal{A}(E)$ is a primitive Mukai vector for some torsion-free $\mathcal{A}$-module $E$ of rank one, then $M_{\mathcal{A}/X}(v)$ is an irreducible symplectic variety deformation equivalent to $Hilb^{\frac{v}{2}+1}(X)$. Moreover we have:
\begin{itemize}
\item $M_{\mathcal{A}/X}(v)\neq \emptyset$ if and only if ${v}^2\geq -2$
\item if ${v}^2=0$, then  $M_{\mathcal{A}/X}(v)$ is a $K3$-surface.
\end{itemize}
\end{cor}
\begin{proof}
Using (\ref{iso}) we have $M_{\mathcal{A}/X}(v)\cong M_H^{Y,G}(v)$. Now the statements follow from (\cite[Theorem 3.16.]{yoshi}).
\end{proof}
\begin{cor}
Assume $v=v_\mathcal{A}(E)$ is a primitive Mukai vector for some torsion-free $\mathcal{A}$-module $E$ of rank one, then we have:
\begin{itemize}
\item $h^{p,q}(M_{\mathcal{A}/X}(v))=h^{p,q}(Hilb^{\frac{v}{2}+1}(X))$
\item $b_i(M_{\mathcal{A}/X}(v))=b_i(Hilb^{\frac{v}{2}+1}(X))$
\end{itemize}
Here $h^{p,q}$ are the Hodge numbers and $b_i$ are the Betti numbers.
\end{cor}
\begin{proof}
This follows from (\ref{main}) using the fact that the Betti numbers and the Hodge numbers are invariant with respect deformation equivalence.
\end{proof}

\addcontentsline{toc}{section}{References}
\bibliography{Artikel}
\bibliographystyle{alphaurl}

\end{document}